\documentclass[a4paper]{article}
\usepackage{amsfonts,amsmath,amssymb,amsthm}
\usepackage[T1]{fontenc}

\usepackage{calc}               

\usepackage{graphicx}           

\usepackage{amsfonts,amsmath,amssymb,amsthm}

\def\={=&\:}
\def\+{+&\:}
\def\-{-&\:}

\newcommand{\nn}{\nonumber}

\newcommand{\scr}{\textsuperscript}

\usepackage{scalerel,stackengine}
\usepackage[shortlabels]{enumitem}

\usepackage{calc,amsmath}
\newlength\dlf

\def\sheaf{\mathcal} 

\usepackage{txfonts} 

\def\C{\mathbb{C}}

\def\N{\mathbb{N}}

\def\Q{\mathbb{Q}}

\def\R{\mathbb{R}}

\def\Z{\mathbb{Z}}

\newcommand{\Map}{\text{Map}}
\newcommand{\rechts}{\:\rightarrow\:}

\DeclareMathOperator{\image}{im}

\usepackage{graphicx} 

\def\const{\text{const.}}
\newcommand{\eps}{\varepsilon}

\newcommand{\reg}{\text{reg.}}

\newtheorem{theorem}{Theorem}
\newtheorem{'theorem'}{''Theorem``}
\newtheorem{lemma}[theorem]{Lemma}
\newtheorem{proposition}{Propos.}
\newtheorem{cor}[proposition]{Corollary}
\newtheorem{definition}[proposition]{Definition}

\newtheorem{definition and theorem}{Definition an Theorem}
\theoremstyle{plain}
\newtheorem*{remark*}{Remark}
\newtheorem*{definition*}{Definition}
\newtheorem*{example*}{Example}

\usepackage{calc}               

\usepackage{graphicx}           

\usepackage{amsmath, amsthm, amssymb}

\usepackage{txfonts} 

\newcommand{\Res}{\text{Res}}

\newlist{mylist}{enumerate*}{1} 
\setlist[mylist]{label=(\roman*)}

\newcommand{\Cspan}{\text{span}_{\C}}

\newcommand\SK{S\hspace{-0.22em}K} 
\newcommand\SL{S\hspace{-0.22em}L} 

\usepackage[mathscr]{eucal}

\usepackage{graphicx} 
\makeatletter
\newcommand{\ostar}{\mathbin{\mathpalette\make@circled*}}
\newcommand{\make@circled}[2]{%
  \ooalign{$\m@th#1\smallbigcirc{#1}$\cr\hidewidth$\m@th#1#2$\hidewidth\cr}%
}
\newcommand{\smallbigcirc}[1]{%
  \vcenter{\hbox{\scalebox{0.77778}{$\m@th#1\bigcirc$}}}%
}
\makeatother

\begin{document}

\title{Convolutions on the complex torus}
\author{Marianne Leitner\\
School of Mathematics, Trinity College, Dublin 2, Ireland\\
School of Theoretical Physics, DIAS, Dublin 4, Ireland\\
\textit{leitner@stp.dias.ie}
}

\maketitle

\begin{abstract}
``Quasi-elliptic'' functions can be given a ring structure in two different ways, 
using either ordinary multiplication, or convolution.
The map between the corresponding standard bases is calculated. 
A related structure has appeared recently in the computation of Feynman integrals.
The two approaches are related by a sequence of polynomials closely tied to the Eulerian polynomials.
\end{abstract}

\section{Introduction}

\subsection{Overview}

Vector bundles on elliptic curves were classified by Atiyah in 1957. 
With respect to their tensor product, they form a ring with two factors,
one related to the prime numbers and one for degree zero \cite[Theorem 12]{A:1957}. 
At that time their sections were of minor significance, 
but recently they became important for problems in quantum field theory and the related area of mock Jacobi forms. 
The sections of the degree zero part are the quasi-elliptic functions studied here. 
It is shown that apart from pointwise multiplication
the space of such functions admits a second algebraic structure that is defined and studied in this paper, namely a convolution. 
Iteratively applied to Eisenstein's zeta function \cite{E:1847}, 
this operation yields a natural basis for a vector space which has applications in number theory and quantum field theory.
A different basis is generated by the Eisenstein-Kronecker function.
The two bases have complementary properties 
and they are related 
by a family of polynomials with intriguing properties. 
As a sequence, they approach zero rapidly in an interval, with a small peak in the middle related to the Bernoulli numbers.

All of these structures may have interesting generalizations, both in the context of elliptic curves and for higher genus.

\subsection{Motivation}

In conformally invariant quantum field theory the study of $N$-point functions and Feynman diagrams on elliptic curves 
requires the calculation of iterated convolutions of Weierstrass elliptic functions and related meromorphic functions of period $1$. 
The first motivation of the present study was to obtain explicit formulae, 
since none are available in standard reference works.

Elliptic polylogarithms can be constructed as iterated integrals of Abelian differentials of the third kind, with simple poles only.
Given an elliptic curve together with $k$ points $x_1,\ldots,x_k$ on it,
they are defined recursively by \cite{BDDT:2018}
\begin{equation}\label{def: elliptic polylogarithm} 
\Gamma\left(\substack{{n_1,\ldots,n_k}\\{x_1,\ldots,x_k}};y\right)
=\int_0^y\:g^{(n_1)}(x-x_1)\:\Gamma\left(\substack{{n_2,\ldots,n_k}\\{x_2,\ldots,x_k}};x\right)\:dx
\end{equation}
where $\Gamma\left(\hspace{0.2cm};y\right)=1$. 
For $n=n_1,\ldots,n_k\in\N$, the integration kernels $g^{(n)}(x)$ are quasi-elliptic. 
Functions with this property are meromorphic on $\C$ with period $1$ but quasi-periodic along the complex period
in the sense that they lie in the kernel of some power of the difference operator. 
A standard example is Eisenstein's zeta function \cite{E:1847}, whose second iterated difference is zero.
In modern terminology, it is defined as the modified Weierstrass zeta function $\mathscr{Z}(x)=\zeta(x)-e_2x$,
where $e_2=\zeta(x+1)-\zeta(x)$ is the quasi-modular Eisenstein series of weight two.
Such functions are sections of vector bundles on an elliptic curve whose transition functions are trivial for shifts by $1$ 
and are of standard Jordan upper triangular form for shifts by another period. 
These degree zero vector bundles needed special attention in Atiyah's classification of vector bundles over an elliptic curve \cite[Theorem 8]{A:1957}. 
Up to isomorphism, an indecomposable bundle is characterized by rank, degree and a point on the curve. 
Atiyah did not extend his investigation to global sections, 
which have become a major topic of number theory in recent years. 
A typical case with non-zero degree is the Appell-Lerch sum $\mu$ that was crucial for Zweger's elucidation of mock theta functions
\cite{Zw:2002}. 
For both $\mu$ and $\mathscr{Z}$ the rank is 2 and the point on the curve is the origin, but for $\mu$ the degree is one.

Degree zero vector bundles of higher rank can be obtained as tensor products of rank two bundles. 
Sections can be described by polynomials in the variable $\mathscr{Z}$, 
but these come with undesirable higher order poles. 
This can be remedied by the use of differential polynomials, 
however this is not a particularly natural procedure.
Instead, formula (\ref{def: elliptic polylogarithm}) suggests to consider convolutions.
Iterated convolutions $\mathscr{Z}^{\ostar  n}(x)$ for $n\geq 1$ have poles of order one only.
 
The convolution $\ostar$ is defined in Section \ref{section: Special quasi-elliptic functions}
and it is shown that it induces a $\C$-bilinear commutative and associative product 
on the space of quasi-elliptic functions.
It gives rise to two alternative filtrations on the above mentioned polynomial ring in $\mathscr{Z}$.

The integration kernels $g^{(n)}(x)$ from eq.\ (\ref{def: elliptic polylogarithm}) 
are generated by the Eisenstein-Kronecker function,
which is a central object in the theory of modular forms and in integrable models.
In Section \ref{section: Application}, 
we specify a base change transformation formula,
which provides a means for computing the values of $\mathscr{Z}^{\ostar  n}(x)$ and its algebraic properties explicitly
(modular transformation behaviour, addition theorems,\ldots).
It also defines a polynomial in $\Q[x]$ of degree $n$.
The change from the functions $g^{(n)}(x)$ to the convolutions $\mathscr{Z}^{\ostar  n}(x)$
has no obvious benefit with regard to modular properties. 
It may be worth mentioning, however, that for $\gamma\in\SL(2,\Z)$, 
$\mathscr{Z}^{\ostar n}(x)|\gamma$ can be expressed as a sum over $g^{(k)}(x)$ for $0\leq k\leq n$
where the coefficient of the $k$\scr{th} term 
is proportional to the $k$\scr{th} derivative of the above mentioned polynomial.
The polynomials $p_n(x)\in\Q[x]$ associated to the converse transformations 
are equivalent to the Eulerian polynomials.
More specifically, for $n\geq 0$, 
$p_{n+1}(x)$ is related to the polylogarithm (Jonqui\`ere's function) of order $-n$
by \cite{T:1945}
\begin{equation*}
\frac{(-1)^n}{n!}\operatorname{Li}_{-n}(x)
=p_{n+1}\left(\frac{1}{1-x}\right)
\:.
\end{equation*}
Though the recurrence relation of the polynomials $p_n(x)$ and the differential equation of their generating function 
are simpler than the corresponding equations for the Eulerian polynomials, 
they do not seem to have been recognized on their own account.

By reference to the functions $g^{(k)}(x)$, iterated convolutions of $\wp(x)$ can be computed according to
\begin{equation*}
\wp^{\ostar n}(x)
=(-1)^n\frac{d^n}{dx^n}\mathscr{Z}^{\ostar  n}(x)+(-1)^ne_2^n
\:.
\end{equation*}
This answers a question raised in \cite{L:2018}.

\subsection{Notations and conventions}

Throughout this paper, $\tau\in\mathfrak{h}$, the complex upper half plane, is fixed.
Let $K_1$ be the set of elliptic functions w.r.t.\ the lattice $\Lambda=\Z+\tau\Z$.
For $k=1,2,3$, let $e_{2k}=2G_{2k}$ be the Eisenstein series of weight $2k$ in the conventions of \cite{Z:1-2-3}.
Eisenstein's zeta function is given in terms of the Weierstrass $\zeta$ function by \cite{E:1847,W:1976}
\begin{equation*}
\mathscr{Z}(x)
=\zeta(x)-e_2 x
\:.
\end{equation*}
For $m,n\in\Z$, Legendre's relation implies
\begin{equation}\label{eq: Z(x+m tau)-Z(x)}
\mathscr{Z}(x+m+n\tau)-\mathscr{Z}(x) 
=-2\pi in
\:.
\end{equation}
Let $\C^{\N}$ have the basis $\mathbf{e}_{\ell}$ for $\ell\in\N$.
Let $f:\C\rechts\C\cup\{\infty\}$ be a meromorophic function of period $1$ with poles only in $\Lambda$. 
For $m\in\Z$, we write the Laurent series expansion of $f$ at $x=m\tau$ in the form 
\begin{equation*}
f(x)
=\sum_{\ell\in\N}\frac{f_{\ell}(m)}{(x-m\tau)^{\ell}}
+\reg
\end{equation*}
where for $\ell\in\N$, $f_{\ell}\in\Map(\Z,\C)$.
We denote the singular part of $f$ by $\sigma(f)\in\Map(\Z,\C^{\N})$, 
\begin{equation*}
\sigma(f)
=\sum_{\ell\in\N}f_{\ell}\mathbf{e}_{\ell}
\:.
\end{equation*}
The subspace of polynomial maps in $\Map(\Z,\C)$ will be denoted by $\sheaf{P}$ or $\C[m]$,
and the corresponding subspace of $\Map(\Z,\C^{\N})$ by $\sheaf{P}^{\N}$.
We write $\sheaf{P}^{\N}=\oplus_{\ell=1}^{\infty}\sheaf{P}^{[\ell]}$ where for $\ell\geq 1$, $\sheaf{P}^{[\ell]}\cong\sheaf{P}$.

Let $f:\C\rechts\C\cup\{\infty\}$ be a meromorophic function of period $1$. 
We define the forward difference operator on $f$ by
\begin{equation*}
\Delta f(x)
:=f(x+\tau)-f(x)
\:,
\quad\forall x\in\C
\:.
\end{equation*}
In the case where $f$ has poles only on $\Lambda$, and $f_{\ell}\in\Map(\Z,\C)$,
we define the forward difference operator on $f_{\ell}$ by
\begin{equation}\label{eq: forward translation by 1 on polynomial}
\Delta_1f_{\ell}(m)
:=f_{\ell}(m+1)-f_{\ell}(m) 
\:.
\end{equation}
We have
\begin{equation}\label{eq: sigma commutes with Delta resp. Delta 1}
\sigma\circ\Delta
=\Delta_1\circ\sigma
\:.
\end{equation}
$\sigma$ and $\Delta$ descend to the space of meromorphic functions modulo additive constants,
where to $f$ corresponds the class $\hat{f}$. 
By abuse of notation,
we use the same letter $\sigma$ and $\Delta$, respectively,
for the corresponding maps on classes of such functions.

\section{Special quasi-elliptic functions}\label{section: Special quasi-elliptic functions}

\subsection{The vector space of special quasi-elliptic functions}

Let $f:\C\rechts\C\cup\{\infty\}$ be a meromorophic function of period $1$. 
For $k\geq 0$, let 
\begin{displaymath}
K_k=\ker\Delta^k 
\end{displaymath}
where $\Delta^k$ denotes the $k$-fold application of $\Delta$
(in particular, $\Delta^0$ is the identity operator).
$f$ will be called quasi-elliptic (w.r.t.\ the lattice $\Lambda$) if $f\in K_k$ for some $k\geq 0$.

For $k=0$, we have $K_0=\{0\}$.
$K_1$ is isomorphic to the field of elliptic functions
\begin{displaymath}
\C(X,Y)/\langle\eps(X,Y)\rangle  
\:,
\end{displaymath}
where $\eps(X,Y):=Y^2-4(X^3-15e_4X-30e_6)$.
The forward difference defines a $K_1$-linear map $\Delta:K_{k+1}\rechts K_{k}$.
The set of quasi-elliptic functions on $\C$ has the structure of a filtered algebra 
\begin{displaymath}
K=\cup_{k\geq 0}K_k 
\end{displaymath}
over the field $K_1$ w.r.t.\ pointwise multiplication.
(So if $f\in K_r$, $g\in K_s$ with $r+s\geq 1$ then $fg\in K_{r+s-1}$.)
Since $\Delta$ descends to $K/\C$, the latter has a filtration by the quotient spaces $K_k/\C$ for $k\geq 1$.
The multiplicative structure is spoiled in the process
($\hat{f}\:\hat{g}\not=\widehat{fg}$ in general, where $\hat{f}$ denotes the equivalence class of $f$ modulo addition of a constant), 
however, the quotient has the convenient feature that for $k\geq 1$, 
$\widehat{\mathscr{Z}^k}\in K_k/\C$.

For the purpose of this paper, $f\in K$ will be called special (w.r.t.\ the lattice $\Lambda$)
if all poles of $f$ are located on points of $\Lambda$.
In this case we write $f\in\SK$.

Since $\sigma$ descends to $K/\C$, the terminology also makes sense for the corresponding classes of functions modulo $\C$.
Let $\SK_k=\SK\cap K_k$.
The special elliptic functions form the ring
\begin{displaymath}
\SK_1
=\C[X,Y]/\langle\eps(X,Y)\rangle  
\:.
\end{displaymath}

\begin{theorem}\label{theorem: Struktursatz for SK functions}(Structure theorem for $\SK$ functions)
\begin{enumerate}[a)]
\item\label{theorem item: SK=SK1[Z]} 
The ring of special quasi-elliptic functions is a polynomial ring in one variable over $\SK_1$,
namely
\begin{displaymath}
\SK
=\SK_1[\mathscr{Z}] 
\:.
\end{displaymath}
Moreover, $\SK$ is a free module over $\SK_1$ with basis $\{\mathscr{Z}^k|\:k\geq 0\}$.
\item\label{theorem item:  functions in SK/C are classified by their principal part, whose Laurent coefficents are polynomials} 
We have
\begin{equation*}
\image(\sigma)
\subseteq\sheaf{P}^{\N}
\:.
\end{equation*}
\item\label{theorem item: isomorphism} 
The map 
\begin{equation}\label{map: SK mod C rechts image (oplus P[k] rechts principal parts at n tau)}
\sigma:\quad
\SK/\C\rechts\sheaf{P}^{\N}
\:
\end{equation}
is an isomorphism.
\end{enumerate}
\end{theorem}

\begin{proof}
\begin{description}
\item[Part \ref{theorem item: SK=SK1[Z]}] 
For $z\in\C$, set $B_0(z)=1$ and $B_1(z)=z$, 
and for $k\geq 1$,
\begin{displaymath}
B_{k+1}(z)
=\frac{1}{k+1}z^{k+1}-\frac{1}{2}z^k+\sum_{m=2}^{k+1}\binom{k+1}{m}B_{m}\:z^{k+1-m}
\end{displaymath}
where $B_m$ for $m\geq 2$ are the Bernoulli numbers.
We define accordingly
\begin{equation}\label{def: An}
A_{n}(x)
:=(\Delta\mathscr{Z})^{n}B_{n}\left(\frac{\mathscr{Z}(x)}{\Delta\mathscr{Z}}\right)
\:.
\end{equation}
So $A_0=1$, $A_1=\mathscr{Z}$ and for $n\geq 1$, 
\begin{equation}\label{eq: tAk as expansion on Z}
A_{n+1}
=\frac{1}{n+1}\mathscr{Z}^{n+1}-\frac{1}{2}(\Delta\mathscr{Z})\mathscr{Z}^{n}+O(\mathscr{Z}^{n-1})
\:.
\end{equation}
In our situation, $\Delta\mathscr{Z}=-2\pi i$.
By the fact that
\begin{displaymath}
B_{n+1}(z+1)-B_{n+1}(z)
=z^n 
\:,
\end{displaymath}
we have
\begin{equation}\label{eq: fiber of Delta over (Delta Z)(Z to the power of n)}
\Delta^{-1}\left((\Delta\mathscr{Z})\mathscr{Z}(x)^n\right) 
=A_{n+1}(x)
+K_1
\:.
\end{equation}
We show that
\begin{equation}\label{eq: SK k spanned over special elliptic functions by Z k' for k'<k}
\SK_k
=\Cspan\{\mathscr{Z}^{k'}|\:k'\leq k-1\}\cdot\SK_1
\: 
\end{equation}
by induction on $k$. For $k=1$ there is nothing to show.
Because of eq.\ (\ref{eq: fiber of Delta over (Delta Z)(Z to the power of n)}),
\begin{equation*}
\SK_{k+1}
=\Cspan\{A_{k'}|\:k'\leq k\}\cdot\SK_1
\:.
\end{equation*}
By eq.\ (\ref{eq: tAk as expansion on Z}),
this implies 
\begin{equation*}
\SK_{k+1}
\subseteq\Cspan\{\mathscr{Z}^{k'}|\:k'\leq k\}\cdot\SK_1
\:.
\end{equation*}
On the other hand,
$\mathscr{Z}^{k'}\in\SK_{k'+1}$.
Thus eq.\ (\ref{eq: SK k spanned over special elliptic functions by Z k' for k'<k}) holds for $k+1$.

The $\mathscr{Z}^k$ for $k\geq 1$ are linearly independent over $\SK_1$.
Otherwise, suppose for $k\geq 1$, $\exists\:\eps_0,\ldots,\eps_k\in\SK_1$ with $\eps_0\not\equiv0$ such that
\begin{equation*}
\eps_0\mathscr{Z}^k+\eps_1\mathscr{Z}^{k-1}+\ldots+\eps_{k-1}\mathscr{Z}
=0\:. 
\end{equation*}
Since for $\ell\geq 1$, $\mathscr{Z}^{\ell}\in K_{\ell+1}$,
it follows that $\eps_0\Delta^{k}\mathscr{Z}^k=0$, contradiction.
\item[Part \ref{theorem item:  functions in SK/C are classified by their principal part, whose Laurent coefficents are polynomials}]
The operator $\Delta$ on $\SK$
gives rise to an operator $\Delta_1$ on the respective Laurent coefficients, cf.\ eq.\ (\ref{eq: forward translation by 1 on polynomial}). 
If $f\in\SK_k$ then for each $\ell$, $\Delta_1^kf_{\ell}=0$.
As in part \ref{theorem item: SK=SK1[Z]}, (replacing $\Delta\mathscr{Z}$ by $\Delta_1m=1$), 
it follows from eq.\ (\ref{eq: fiber of Delta over (Delta Z)(Z to the power of n)}) 
that $f_{\ell}$ is a polynomial, of order $\leq k-1$.
Suppose $f_{\ell}\not\equiv0$ for infinitely many $\ell\in\N$.
Then $\exists m\in\{0,1,\ldots,k\}$ such that $f_{\ell}(m)\not=0$ for infinitely many $\ell$.
So $f$ is not meromorphic, contradiction.
\item[Part \ref{theorem item: isomorphism}] 
We show that the map (\ref{map: SK mod C rechts image (oplus P[k] rechts principal parts at n tau)}) is injective. 
Indeed, suppose $f\in\SK$ has no singularities.
By part \ref{theorem item: SK=SK1[Z]},
we have $f=\sum_{\ell=0}^ka_{\ell}\mathscr{Z}^{\ell}$ with $a_0,\ldots,a_n\in\SK_1$.
If $k=0$ then $f=a_0\in\SK_1$, so its image in $\SK/\C$ vanishes.
Suppose $k\geq 1$, and suppose $a_k\not=0$.
Since $\mathscr{Z}^{\ell}\in\SK_{\ell+1}$, $\Delta^kf=a_k\Delta^k\mathscr{Z}^k$ is regular by assumption on $f$, 
so $a_k=\const$
It follows that $\Delta^{k-1}f=a\mathscr{Z}+b$ with $a\propto a_{k}$ and $b\propto a_{k-1}$.
But $\Delta^{k-1}f$ is regular, so $a_{k-1}$ has a single simple pole, contradiction to $a_{k-1}\in\SK_1$.
We conclude that $f\in\C$, so $\sigma$ is injective.

$\sheaf{P}^{\N}$ has basis $\{m^k\mathbf{e}_{\ell}|\:k\geq 1\}$.
Differentiation in $\SK/\C$ gives rise to an operator $D:\sheaf{P}^{\N}\rechts\sheaf{P}^{\N}$
with
\begin{equation*}
\sigma\circ\frac{d}{dx}
=D\circ\sigma
\:.
\end{equation*}
On the direct summands, it is given by the family of operators 
\begin{equation*}
D_\ell:\quad\sheaf{P}^{[\ell]}\rechts\sheaf{P}^{[\ell+1]}
\:,\quad\ell\geq 1\:,
\end{equation*}
which map $p_{\ell}\mathbf{e}_{\ell}\in\sheaf{P}^{[\ell]}$ to $-\ell p_{\ell}\mathbf{e}_{\ell+1}\in\sheaf{P}^{[\ell+1]}$.
Since $\sigma(\mathscr{Z})=\mathbf{e}_1$, 
we have for $\ell\geq 1$,
\begin{equation}\label{eq: sigma(ell'th order derivative of Z)}
\sigma(\frac{d^{\ell}}{dx^{\ell}}\mathscr{Z})
=(-1)^{\ell}\ell!\mathbf{e}_{\ell+1}
\:.
\end{equation}
So it suffices to consider $\ell=1$ and the basis in $\sheaf{P}^{[1]}$.
Suppose that for $1\leq k\leq N$, $\exists \hat{f}^{(k)}\in\SK/\C$ such that
\begin{equation*}
\sigma\left(\hat{f}^{(k)}\right)
=m^k\mathbf{e}_1
\:.
\end{equation*}
Then $\exists\:\hat{f}^{(N+1)}\in\SK/\C$ such that
\begin{equation*}
\Delta\hat{f}^{(N+1)}
=\hat{f}^{(N)}
\:.
\end{equation*}
By eq.\ (\ref{eq: sigma commutes with Delta resp. Delta 1}),
we have
\begin{equation*}
\Delta_1(\sigma(\hat{f}^{(N+1)}))
=m^N\mathbf{e}_1
\:.
\end{equation*}
By eq.\ (\ref{eq: fiber of Delta over (Delta Z)(Z to the power of n)}) 
(with $\mathscr{Z}$ replaced by $m$, where $\Delta_1m=1$),
there exist $c_1,c_2,\ldots\in\C$ such that
\begin{equation*}
\sigma(f^{(N+1)})
=B_{N+1}(m)\mathbf{e}_1
+\sum_{\ell\geq 1}c_{\ell}\:\mathbf{e}_{\ell}
\:.
\end{equation*}
By the induction hypothesis, $\exists\:\hat{f}^{(1)},\ldots,\hat{f}^{(N)}\in\SK/\C$ such that
\begin{equation*}
B_{N+1}(m)\mathbf{e}_1
=\frac{m^{N+1}}{N+1}\mathbf{e}_1+\sum_{i=0}^Nh_i\:\sigma(\hat{f}^{(i)})
\:.
\end{equation*}
Using eq.\ (\ref{eq: sigma(ell'th order derivative of Z)}),
we conclude that 
\begin{equation*}
\frac{m^{N+1}}{N+1}\mathbf{e}_1
=\sigma\left(\hat{f}^{(N+1)}-\sum_{i=0}^Nh_i\hat{f}^{(i)}-\sum_{\ell\geq 1}c_{\ell}\:(-1)^{\ell-1}(\ell-1)!\frac{d^{\ell-1}}{dx^{\ell-1}}\mathscr{Z}\right)
\:,
\end{equation*}
so $\sigma$ is surjective.
\end{description}
\end{proof}

%

For $a\in\C$, let $T_a$ be the backward translation by $a$, $T_af(x):=f(x-a)$ for $f\in K$.

\begin{cor}
For $k\geq 1$, $T_a\SK_k/\C$ is defined for $a\in\C/\Lambda$, 
and we have
\begin{equation*}\label{map: isomorphism between all quasi-elliptic functions mod C and the direct sum of all translated special quasi-elliptic functions mod C} 
K_k/\C
=\oplus_{a\in\C/\Lambda}T_a\SK_k/\C\:.
\end{equation*}
\end{cor}

\begin{proof}
Let $a\in\C$ lie in the unit cell of $\Lambda$.
For the translated lattice $a+\Lambda$,
Theorem \ref{theorem: Struktursatz for SK functions}
yields $T_a\SK=T_a\SK_1[T_a\mathscr{Z}]$,
where
\begin{equation}\label{map: TaSk mod C isomorphic to P N}
T_a\SK/\C\cong\sheaf{P}^{\N} 
\:.
\end{equation}
(Note that the isomorphism does not depend on the specific representative of $a$ modulo $\Lambda$.)
A monomorphism 
\begin{equation}\label{map: isomorphism between the direct sum of all translated special quasi-elliptic functions mod C and all quasi-elliptic functions mod C} 
\oplus_{a\in\C/\Lambda}T_a\SK_k/\C\rechts K_k/\C
\:
\end{equation}
is given by replacing the direct sum of elements in $T_a\SK_k/\C\subseteq K/\C$ by their sum.
(\ref{map: isomorphism between the direct sum of all translated special quasi-elliptic functions mod C and all quasi-elliptic functions mod C}) 
is also an epimorphism:
The restriction map $\C\rechts a+\Lambda$ induces a projection $\pi_a$
from the set of the Laurent series expansions of elements $\hat{f}\in K_k/\C$ at $x\in\C$ to those at $a+\Lambda$.
By identifying a meromorphic function with the set of its Laurent series expansions, 
$\pi_a$ maps $\hat{f}\in K_k/\C$ to a meromorphic function modulo $\C$, which we denote by $\hat{f}^{(a)}$.
Since
\begin{equation*}
\pi_a\circ\Delta
=\Delta\circ\pi_a
\:,
\end{equation*}
we have $\hat{f}^{(a)}\in T_a\SK_k/\C$. 
In particular, by the isomorphism (\ref{map: TaSk mod C isomorphic to P N}),
the Laurent coefficent $\hat{f}^{(a)}_{\ell}$ of $\hat{f}^{(a)}$ for a pole of order $\ell\geq 1$
is a polynomial of order $\leq k-1$.
Since $\hat{f}^{(a)}$ is meromorphic, the argument employed in the proof of Theorem \ref{theorem: Struktursatz for SK functions}, part \ref{theorem item:  functions in SK/C are classified by their principal part, whose Laurent coefficents are polynomials} shows that 
for every $\ell\geq1$, $\hat{f}^{(a)}_{\ell}\equiv 0$ for almost all $a$ in the unit cell.
Since there are no nonconstant regular functions on $\C$, it follows that $\hat{f}^{(a)}\equiv 0$ for almost all $a$ in the unit cell of $\Lambda$.
We conclude that to every $\hat{f}\in K/\C$ is associated an element in $\oplus_{a\in\C/\Lambda}T_a\SK_k/\C$.
So (\ref{map: isomorphism between the direct sum of all translated special quasi-elliptic functions mod C and all quasi-elliptic functions mod C}) is surjective.
\end{proof}

Any function in $\SK$ can be characterized by its class in $\SK/\C$ together with its value under contour integration
\begin{displaymath}
\oint:\quad\SK\rechts\C 
\end{displaymath}
along the real period.
Note that
\begin{equation*}
\oint\circ\frac{d}{dx}
=0
\end{equation*}
so that it suffices to treat functions with first order poles only.

\subsection{Convolutions on the complex torus}

\begin{definition}
Let $f,g:\C\rechts\C\cup\{\infty\}$ be meromorphic, $1$-periodic functions having poles only in $\Lambda$.
For $0<\Im(x)<2\Im(\tau)$, we define 
\begin{equation}\label{eq: def of f ostar g}
(f\ostar_+g)(x) 
:=\oint_{\Im(z)=\frac{\Im(x)}{2}}f(x-z)g(z)dz
\:.
\end{equation}
\end{definition}

By Cauchy's Theorem,$f\ostar_+ g$ has a unique analytic continuation, 
which is meromorphic on $\C$ and has period $1$. 
By abuse of notation, we also write $f\ostar_+ g$ for the analytic continuation of $f\ostar_+ g$.

\begin{proposition}
For $f,g\in\SK$, we have the product rule (for the convolution $\ostar_+$)
\begin{equation}\label{eq: Delta(f ostar+ g) with Delta moved into first factor}
\Delta(f\ostar_+g)(x)
=(\Delta f\ostar_+g)(x)+2\pi i\Res_{z=\tau}\left[f(z)g(x+\tau-z)\right]\:,
\quad
x\in\C\setminus\Lambda\:.
\end{equation}
\end{proposition}

\begin{proof}
Suppose $\Im(x)=\eps$ with $0<\eps<\Im(\tau)$.
We have
\begin{equation*}
\Delta(f\ostar_+g)(x)
=\oint_{\Im(z)=\frac{\Im(\tau)}{2}+\frac{\eps}{2}}f(x+\tau-z)g(z)dz
-\oint_{\Im(z)=\frac{\eps}{2}}f(x-z)g(z)dz
\:.
\end{equation*}
By deforming the contour of the first integral into that of the second, 
the imaginary part of the argument of $f$ changes from a value $<\Im(\tau)$ to a value $>\Im(\tau)$,
so $f$ crosses a singularity for $\Im(z)=\Im(x)$, while $g$ remains regular throughout.
This yields
\begin{equation*}
\oint_{\Im(z)=\frac{\Im(\tau)}{2}+\frac{\eps}{2}}f(x+\tau-z)g(z)dz
=\oint_{\Im(z)=\frac{\eps}{2}}f(x+\tau-z)g(z)dz
-2\pi i\Res_{z=x}\left[f(x+\tau-z)g(z)\right] 
\:.
\end{equation*}
The variable transformation $z\mapsto x+\tau-z$ 
yields
\begin{equation*}
\Res_{z=x}\left[f(x+\tau-z)g(z)\right]
=-\Res_{z=\tau}\left[f(z)g(x+\tau-z)\right]
\:.
\end{equation*}
This proves the equation for $0<\Im(x)<\Im(\tau)$.
By analytic continuation, it holds for $x\in\C\setminus\Lambda$.
\end{proof}

\begin{lemma}\label{lemma: The set of quasi-elliptic functions modulo C defines a filtered algebra w.r.t. the convolution}
The set of special quasi-elliptic functions modulo $\C$ defines a filtered algebra 
\begin{displaymath}
\SK/\C
=\cup_{k\geq 0}\:\SK_k/\C 
\end{displaymath}
w.r.t.\ the convolution $\ostar_+$.
(So if $\hat{f}\in\SK_r/\C$ and $\hat{g}\in\SK_s/\C$ then $\hat{f}\ostar_+ \hat{g}\in\SK_{r+s}/\C$.) 
\end{lemma}

\begin{proof}
For $i=1,2$, let $\hat{f}_i\in\SK_{k_i}/\C$.
We have to show that
\begin{equation}\label{inclusion: hat fk1 ostar+ hat fk2 in SK(k1+k2) mod C}
\hat{f}_1\ostar_+\hat{f}_2\in\SK_{k_1+k_2}/\C 
\:.
\end{equation}
We may assume that both $k_1,k_2\geq 1$ since the inclusion (\ref{inclusion: hat fk1 ostar+ hat fk2 in SK(k1+k2) mod C}) is trivial otherwise.
We proceed by induction on $k_1+k_2$.
Suppose that for $2\leq k_1+k_2\leq N$, (\ref{inclusion: hat fk1 ostar+ hat fk2 in SK(k1+k2) mod C}) holds.
Let $k_1+k_2=N+1$.
Since $\Delta\hat{f}_1\in\SK_{k_1-1}/\C$,
we have $(\Delta \hat{f}_1)\ostar_+ \hat{f}_2\in\SK_N/\C$ by hypothesis,
and by eq.\ (\ref{eq: Delta(f ostar+ g) with Delta moved into first factor}) 
it remains to show that $\Res_{z=\tau}[\hat{f}_1(z)\hat{f}_2(x+\tau-z)]\in\SK_N/\C$ for $x\not\in\Lambda$.

Denote by $o_{\hat{f}_1}(z)$ the pole order of $\hat{f}_1$ at the point $z$.
If $o_{\hat{f}_1}(\tau)=0$, the proof is complete.
For $o_{\hat{f}_1}(\tau)=1$, the residue is a constant multiple of $\hat{f}_2(x)\in\SK_N/\C$. 
If $o_{\hat{f}_1}(\tau)>1$, the residue is a complex polynomial in $z-\tau$ whose coefficients are derivatives of $\hat{f}_2$ of order $\leq (o_f(\tau)-1)$, at position $x$.
But $\Delta$ commutes with differentiation,
\begin{equation*}
\Delta\circ\frac{d}{dx}
=\frac{d}{dx}\circ\Delta 
\end{equation*}
so residue lies in $\SK_N/\C$ again.
We conclude that $\Delta(\hat{f}_1\ostar_+ \hat{f}_2)\in \SK_{N}/\C$, 
so the proof of (\ref{inclusion: hat fk1 ostar+ hat fk2 in SK(k1+k2) mod C}) is complete.
\end{proof}

\begin{theorem}\label{theorem: basic properties of the product ostar}
The convolution $\ostar_+$ defines a $\C$-bilinear product $\SK\times\SK\rechts\SK$ 
with the following properties: For $f,g,h\in\SK$ we have
\begin{enumerate}
 \item (Commutativity:) $f\ostar_+ g=g\ostar_+ f$.
 \item (Associativity:) $(f\ostar_+ g)\ostar_+ h=f\ostar_+(g\ostar_+ h)$.
 \item\label{theorem item: regularity} (Regularity:) 
 Suppose for $n,m\in\N$, 
 $f$ is regular on $0<\Im(x)<m\Im(\tau)$, and $g$ is regular on $0<\Im(x)<n\Im(\tau)$.
 Then $f\ostar_+ g$ is regular on $0<\Im(x)<(m+n)\Im(\tau)$.
 \end{enumerate}
\end{theorem}

\begin{proof}
From eq. (\ref{eq: Delta(f ostar+ g) with Delta moved into first factor}) follows $f,g\in\SK$ $\Rightarrow$ $f\ostar_+ g\in\SK$. 
$\C$-bilinearity is manifest from the defining eq.\ (\ref{eq: def of f ostar g}).
\begin{enumerate}
\item
Let $\delta(x)$ be the Dirac delta distribution on $\R$. 
After substitution $x_1=\Re(x-z)$, $y_1=\Im(x-z)$ and $x_2=\Re(z)$, $y_2=\Im(z)$,
eq.\ (\ref{eq: def of f ostar g}) reads
\begin{equation}\label{eq: symmetric def of f ostar g}
(f\ostar_+ g)(x)
=\int_0^1\int_0^1f(x_1+iy_1)g(x_2+iy_2)\delta(x_1+x_2-\Re(x))\:dx_1dx_2  
\end{equation}
for $y_1=y_2=\frac{\Im(x)}{2}$, so the product is symmetric.
More generally, we can put $0<y_1,y_2<\Im(\tau)$ and $y_1+y_2=\Im(x)$. 
\item 
For $f,g,h\in\SK$, and for $0<y<2\Im(\tau)$,
we have
\begin{equation*}
((f\ostar_+ g)\ostar_+ h)(x+iy)
=\int_0^1\int_0^1(f\ostar_+ g)(z_a)h(z_b)\:\delta(x_a+x_b-x)\:dx_adx_b
\end{equation*}
where for $\iota=a,b$, we have $z_{\iota}=x_{\iota}+iy_{\iota}$ with $x_{\iota}\in\R$ and $y_{\iota}=y/2$. 
But
\begin{align*}
((f\ostar_+ g)\ostar_+ h)(x+iy)
\=\int_0^1\int_0^1\int_0^1f(z_1)g(z_2)h(z_b)\:\delta(x_1+x_2-x_a)\times\\
&\delta(x_a+x_b-x)\:dx_1dx_2dx_b
\:,
\end{align*} 
where for $j=1,2$, we have $z_j=x_j+iy_j$ with $x_j\in\R$ and $y_j=\frac{y_b}{2}=y/4$, and so $y_1+y_2+y_b=y$
(with the corresponding equation for the real parts).
The range can be symmetrized to $y_1=y_2=y_b=\frac{y}{3}$ without the intergrand crossing a singularity,
since $f,g,h$ are special and $0<y_1,y_2,y_b<\Im(\tau)$. 
\item This follows from eq.\ (\ref{eq: symmetric def of f ostar g}).
Indeed, $f\ostar_+ g$ can be analytically continued to $0<\Im(x)<(n+m)\Im(\tau)$ without the integral acquiring residue terms.
\end{enumerate}
\end{proof}

\begin{cor}\label{cor: n-fold convolution of functions in SK, written symmetrically} 
Let $n\geq 2$ and let $f_1,\ldots,f_n\in\SK$.
For $x,y\in\R$ with $0<y<n\Im(\tau)$,
we have
\begin{equation*}
(f_1\ostar_+\ldots\ostar_+ f_n)(x+iy) 
=\int_0^1\ldots\int_0^1f_1(x_1+iy_1)\ldots f_n(x_n+iy_n)\:\delta(\sum_{i=1}^nx_i-x)\:dx_1\ldots dx_n
\end{equation*}
with $y_1=\dots=y_n=\frac{y}{n}$.
\end{cor}

For $f\in\SK$ and the product $f\ostar_+\ldots\ostar_+ f$ (with $n\geq 2$ factors), 
we shall also write $f^{\ostar_+ n}$.
We set $f^{\ostar_+ 0}=1$ and $f^{\ostar_+ 1}=f$.
Thus for $m,n\geq 0$, we have $f^{\ostar_+ n}\ostar_+ f^{\ostar_+ m}=f^{\ostar_+ (n+m)}$.

Another useful identity is $\frac{d}{dx}(f\ostar_+g)=(\frac{d}{dx}f)\ostar_+g=f\ostar_+(\frac{d}{dx}g)$.

\begin{theorem}\label{theorem: pole order of f ostar g leq o(f)+o(g)-1}
Let $f,g:\C\rechts\C\cup\{\infty\}$ be meromorphic and $1$-periodic functions. 
Define a meromorphic function $f\ostar g$ by the integral in eq.\ (\ref{eq: def of f ostar g}).
\begin{enumerate}[a)]
 \item\label{theorem item: pole set of convolution of two meromorphic 1-periodic functions} 
 Let $a_I=\{a_i\}_{i\in I}$ and $b_J=\{b_j\}_{j\in J}$ be the pole sets of $f$ and of $g$, respectively.
 The pole set of $f\ostar g$ is $a_I+b_J=\{a_i+b_j|a_i\in a_I,\:b_j\in b_J\}$.
\item\label{theorem item: pole order of convolution of two meromorphic 1-periodic functions} 
If $f$ and $g$ have maximal pole order $o_f$ and $o_g$, respectively, where $o_f+o_g\geq 1$,
then the integral in eq.\ (\ref{eq: def of f ostar g}) has maximal pole order
\begin{equation*}
o_{f\ostar g}\leq o_f+o_g-1 
\:.
\end{equation*}
\item\label{theorem item: filtration on K defined by the pole order}  
The set $K$ of quasi-elliptic functions has the structure of a filtered algebra over $\C$ w.r.t.\ $\ostar$,
defined by the pole order.
\end{enumerate}
\end{theorem}


\begin{proof}
\begin{description}
 \item[Part \ref{theorem item: pole set of convolution of two meromorphic 1-periodic functions}]
 $f(x-z)g(z)$ has poles at $x-z\in a_I$ and at $z\in b_J$, so the pole set of $(f\ostar g)(x)$ equals $a_I+b_J$.
\item[Part \ref{theorem item: pole order of convolution of two meromorphic 1-periodic functions}]
In the notations from part \ref{theorem item: pole set of convolution of two meromorphic 1-periodic functions},
suppose that $I,J=\N$ and that for $n\geq 1$, we have $\Im(a_n)<\Im(a_{n+1})$ and $\Im(b_n)<\Im(b_{n+1})$.
Let $x,y\in\R$. 
$(f\ostar g)(x+iy)$ is regular on 
\begin{equation}\label{ineq:}
\Im(a_1)+\Im(b_1)
<y
<\Im(a_2)+\Im(b_2) 
\:.
\end{equation}
Let $y=y_a+y_b$ where $\Im(a_1)<y_a<\Im(a_2)$ and $\Im(b_1)<y_b<\Im(b_2)$.
We have
\begin{equation*}
(f\ostar g)(x+iy)
=\int_0^1f(x-x'+iy_a)g(x'+iy_b)dx'
\:.
\end{equation*}
For $\eps>0$ small enough, 
the integral is regular for $y_a=\Im(a_2)-\eps$ and $y_b=\Im(b_2)-\eps$ ,
and we have
\begin{align*}
(f\ostar g)(x+iy)
\=\int_0^1f(x-x'+i(y_a-2\eps))g(x'+i(y_b+2\eps))dx'\\
&+2\pi i\:\Res_{x'+iy_b=b_2}\left[f(x-x'+iy_a)g(x'+iy_b)\right]
\end{align*}
Since $y=y_a+y_b=\Im(a_2)+\Im(b_2)-2\eps$ as before,
inequality (\ref{ineq:}) holds and the integral is regular. 
Moreover, it can be analytically continued to $y<\Im(a_2)+\Im(b_3)$.

By induction, $\sigma(f\ostar g)$ is given by the poles of its residues.
It suffices to show that for $f,g$ with $o_f+o_g\geq 1$, we have
\begin{displaymath}
o_{\Res_{w=w_0}[f(z-w)g(w)]}
\leq o_f+o_g-1
\:.
\end{displaymath}
This follows from the proof of Lemma \ref{lemma: The set of quasi-elliptic functions modulo C defines a filtered algebra w.r.t. the convolution}.
 \item[Part \ref{theorem item: filtration on K defined by the pole order}]
 This is a direct consequence of part \ref{theorem item: pole order of convolution of two meromorphic 1-periodic functions}.
\end{description}
\end{proof}

There is another variant of a convolution on $\SK$, which is regular on $\R$ though it is non-commutative and non-associative. 
For $f,g\in\SK$ and for $x\in\R$, we define
\begin{align*}
(f*_+g)(x)
=\lim_{\eps\searrow 0}\oint_{\Im(z)=-\eps}f(x-z)g(z)dz\:. 
\end{align*}
By abuse of notation, we continue to write $f*_+g$ for the analytic continuation. 

\begin{proposition}
Let $f,g\in\SK$.
For $\Im(x)>0$ sufficiently small, 
we have
\begin{equation*}
(f*_+g)(x)-(f\ostar_+ g)(x)
=2\pi i\Res_{z=0}[f(x-z)g(z)]
\:.
\end{equation*}
\end{proposition}

\begin{proof}
Let $\Im(x)=2\eps$, where $0<\eps<\frac{\Im(\tau)}{3}$.
Then the argument of $f$ in the integrand of $(f*_+g)(x)$ and of $(f\ostar_+ g)(x)$
has an imaginary part equal to $3\eps$ and $\eps$, respectively.
The two can be moved into one another without $f$ crossing a singularity.
On the other hand, by changing $\Im(z)$ from $-\eps$ to $\frac{\Im(x)}{2}=\eps$,
$g$ crosses a singularity at $z=0$.
\end{proof}

\subsection{The subspace of functions in $\SK$ with at most simple poles}

\begin{theorem}\label{theorem: V has basis given by Z oreg n} 
Let $V=\{f\in\SK|\:o_f\leq 1\}$, where $o_f$ is the maximal pole order of $f$.
\begin{enumerate}[a)]
\item\label{theorem item: V spanned by Z ostar n} 
$V$ is a complex vector space, and we have $V/\C=\Cspan\{\widehat{\mathscr{Z}^{\ostar_+ n}}|\:n=1,2,\ldots\}$.
\item\label{theorem item: Res Z ostar n} 
For $n\geq 1$, and for $m\in\N_0$, the residue of $\mathscr{Z}^{\ostar_+ n}(x)$ at $x=m\tau$ is
\begin{equation}\label{eq: Laurent series coefficient of Z oreg N to first order pole at u tau}
f^{[n]}_1(m)
=(-\Delta\mathscr{Z})^{n-1}\binom{m-1}{n-1}
\:.
\end{equation}
\end{enumerate}
\end{theorem}

\begin{proof}
Clearly $V$ is a $\C$ linear space. 
We show that $\mathscr{Z}^{\ostar_+n}$ for $n\geq 0$ define elements of $V$.
By $1$-periodicity, it suffices to consider the poles in $\tau\Z$.
For $n=0$ there is nothing to show. 
For $n=1$, we have $\mathscr{Z}\in\SK_2$ 
and by eq.\ (\ref{eq: Z(x+m tau)-Z(x)}),
the simple pole of $\mathscr{Z}$ at $x=0$ gives rise to a simple pole on every lattice point.
For $n\geq 2$, $\mathscr{Z}^{\ostar_+n}$ is regular at $x=m\tau$ for $1\leq m\leq n-1$,
by Corollary \ref{cor: n-fold convolution of functions in SK, written symmetrically}.
By part \ref{theorem item: pole order of convolution of two meromorphic 1-periodic functions} of Theorem \ref{theorem: pole order of f ostar g leq o(f)+o(g)-1},
we have $o_{\mathscr{Z}^{\ostar_+n}}\leq 1$.
Thus for $n\geq 0$, $\mathscr{Z}^{\ostar_+n}\in V$.
Conversely, we show that the corresponding classes of functions modulo $\C$ span $V/\C$. 
Let $\hat{f}\in V/\C$ satisfy $\Delta^k\hat{f}=0$. 
The case $k=0$ is trivial.
For $k=1$, eq.\ (\ref{eq: fiber of Delta over (Delta Z)(Z to the power of n)}) (with $n=0$)
yields $\hat{f}=a\hat{\mathscr{Z}}$ for some $a\in\C$.
Suppose that for $N\geq 1$,
\begin{equation*}
(V\cap K_N)/(V\cap K_{N-1})
=\Cspan\{\widehat{\mathscr{Z}^{\ostar_+N}}\}
\:.
\end{equation*}
By eq.\ (\ref{eq: Delta(f ostar+ g) with Delta moved into first factor}),
for $n\geq 1$,
\begin{equation}\label{eq: Delta Z ostar+ n}
\Delta\mathscr{Z}^{\ostar_+ (n+1)}(x)
=(\Delta\mathscr{Z})\left\{1\ostar_+\mathscr{Z}^{\ostar_+n}-\mathscr{Z}^{\ostar_+n}(x)\right\}
\:.
\end{equation}
For $f=\mathscr{Z}$ and $g=1\ostar_+\mathscr{Z}^{\ostar_+n}$,
the product rule $\Delta(fg)=(\Delta f)g+f(\Delta g)+(\Delta f)(\Delta g)$ yields $(\Delta f)g=\Delta(fg)$.
So $\Delta^{-1}\left((\Delta\mathscr{Z})\widehat{\mathscr{Z}^{\ostar_+N}}\right)\in\widehat{\mathscr{Z}^{\ostar_+ (N+1)}}+\C\widehat{\mathscr{Z}}$. 
This proves part \ref{theorem item: V spanned by Z ostar n}.

By Theorem \ref{theorem: Struktursatz for SK functions} and by part \ref{theorem item: V spanned by Z ostar n}, 
the Laurent coefficient $f^{[n]}_1(m)$ of the first order pole of $\mathscr{Z}^{\ostar_+ n}$ at position $m\tau$ 
is proportional to $\prod_{k=1}^{n-1}(m-k)=(n-1)!\binom{m-1}{n-1}$.
From eq.\ (\ref{eq: Delta Z ostar+ n}) follows by induction that for $n\geq 1$,
$\Delta^{n-1}\widehat{\mathscr{Z}^{\ostar_+ n}}(x)
=(-\Delta\mathscr{Z})^{n-1}\widehat{\mathscr{Z}}$.
On the other hand, for $m\geq 0$, $\Delta^{m}=\sum_{k=0}^{m}\binom{m}{k}(-1)^{m-k}T_{-k\tau}$, 
so
\begin{equation*}
(-\Delta\mathscr{Z})^{n-1}=\Res_{x=\tau}\left[\Delta^{n-1}\mathscr{Z}^{\ostar_+ n}(x)\right] 
=\Res_{x=0}\left[\mathscr{Z}^{\ostar_+n}(x+n\tau)\right] 
=f^{[n]}_1(n)
\:.
\end{equation*}
We conclude that the proportionality factor equals $(-\Delta\mathscr{Z})^{n-1}/(n-1)!$.
\end{proof}

For $n\geq 2$, we have $\Delta\mathscr{Z}^{\ostar_+(n+1)}=\Delta\mathscr{Z}^{\ostar_+n}\ostar_+\mathscr{Z}$ 
by eq.\ (\ref{eq: Delta(f ostar+ g) with Delta moved into first factor}),
so by comparison with eq.\ (\ref{eq: Delta Z ostar+ n}), we obtain
$(1\ostar_+\mathscr{Z}^{\ostar_+n})=(1\ostar_+\mathscr{Z})^n$, 
which is Fubini's Theorem.

\begin{cor}\label{cor: SK is a polynomial ring in wp and wp' with coefficients which are powers of Z, and alternative bases given}
The ring $\SK/\C$ is generated by any of the following classes of functions modulo $\C$:
\begin{enumerate}[a)]
 \item\label{cor item: derivatives of iterated convolutions of Z} 
 $\widehat{\frac{d^k}{dx^k}\mathscr{Z}^{\ostar_+ n}}$
  \item\label{cor item: derivatives of powers of Z}  
 $\widehat{\frac{d^k}{dx^k}\mathscr{Z}^{n}}$
 \item\label{cor item: powers of Z and their product with derivatives of Weierstrass p}  
 $\widehat{\mathscr{Z}^{n}}$, $\widehat{\mathscr{Z}^{n}\frac{d^k}{dx^k}\wp}$
 \item\label{cor item: derivatives of Weierstrass p and their product with tAn}  
 $\widehat{A_{n-1}\frac{d^k}{dx^k}\wp}$ \hspace{0.3cm}for $A_n$ defined by eq.\ (\ref{def: An}),
\end{enumerate}
where $n=1,2,\ldots$ and $k=0,1,2,\ldots$ 
\end{cor}

\begin{proof}
Part \ref{cor item: derivatives of iterated convolutions of Z} 
is a consequence of Theorem \ref{theorem: V has basis given by Z oreg n}.
Part \ref{cor item: powers of Z and their product with derivatives of Weierstrass p}
is a reformulation of Theorem \ref{theorem: Struktursatz for SK functions}.\ref{theorem item: SK=SK1[Z]},
and implies part \ref{cor item: derivatives of powers of Z}.
Part \ref{cor item: derivatives of Weierstrass p and their product with tAn} 
follows from part \ref{cor item: powers of Z and their product with derivatives of Weierstrass p}
and from the fact that $\SK=\Cspan\{A_n|\:n\geq 0\}\cdot\SK_1$, cf.\ the proof of Theorem \ref{theorem: Struktursatz for SK functions}.
\end{proof}

\section{Application}\label{section: Application}

\subsection{Relation to the set of integration kernels from eq.\ (\ref{def: elliptic polylogarithm})}

For $x,y\in\C$, let
\begin{equation}\label{def: F(x,y)}
F(x,y)
=\frac{\vartheta'(0)\vartheta(x+y)}{\vartheta(x)\vartheta(y)}
\:,
\end{equation}
where $\vartheta(x)$ is the Jacobi theta function \cite[p.\ 311]{K:1881} and \cite{Z:1991}.
The function $F(x,y)$ was called Eisenstein-Kronecker function by \cite{BDDT:2018}, with good justification \cite{Levin:1997}.
In the following we fix $y\not\in\Lambda$ and consider $F(x,y)$ as a function of $x$.
$F(x,y)$ is $1$-periodic and satisfies
\begin{equation*}
F(x+\tau,y)
=F(x,y)\exp\left((\Delta\mathscr{Z})y\right) 
\:.
\end{equation*}
$F(x,y)$ has a simple pole at $x=m\tau+s$ ($m,s\in\Z$) of residue $\exp((\Delta\mathscr{Z})my)$ 
and is holomorphic on $\C\setminus\Lambda$ \cite{Z:1991}.
Let $g^{(n)}(x)$ be a  generating sequence of $F(x,y)$,
\begin{displaymath}
F(x,y)
=\frac{1}{y}\sum_{n\geq 0}g^{(n)}(x)y^n
\:.
\end{displaymath}
The properties mentioned for $F(x,y)$ imply 
that $g^{(n)}(x)$ for $n\geq 0$ is $1$-periodic and satisfies
\begin{equation}\label{eq: Delta g(k) as sum over g(k-ell)}
\Delta g^{(n)}(x)
=\sum_{k=1}^{n}\frac{(\Delta\mathscr{Z})^k}{k!}g^{(n-k)}(x)
\:.
\end{equation}
Moreover, for $n\geq 1$ and $m\in\Z$, 
the residue of $g^{(n)}(x)$ at $x=m\tau$ is
\begin{equation}\label{eq: residue of g(n) at m tau}
g_1^{[n]}(m)
=\frac{m^{n-1}}{(n-1)!}(\Delta\mathscr{Z})^{n-1}
\:.
\end{equation}
Note that $g^{(n)}$ is regular on $\R$, except for $n=1$.\footnote{\label{0 to power 0}We use the convention $0^0=1$.}
The first few functions are given by
\begin{align*}
g^{(0)}(x)
\=1\\
g^{(1)}(x)
\=\mathscr{Z}(x)\\
g^{(2)}(x)
\=\frac{1}{2}\mathscr{Z}(x)^2-\frac{1}{2}\wp(x)\\
g^{(3)}(x)
\=\frac{1}{6}\mathscr{Z}(x)^3
-\frac{1}{2}\wp(x)\mathscr{Z}(x)
-\frac{1}{6}\frac{d}{dx}\wp(x)
\:.
\end{align*}
The functions $g^{(n)}$ form a set of linearly independent integration kernels
for constructing elliptic polylogarithms as iterated integrals \cite{BDDT:2018}. 
For low $n\geq 2$, it is easy to verify from the above formulae that all poles on $\R$ drop out.
A better adapted approach, however, is to start from the equation
\begin{equation}\label{eq: g(n) expressed in terms of Z ostar k}
g^{(n)}(x)
=\sum_{k=0}^nC^{[n]}_k\mathscr{Z}^{\ostar_+k}(x)
\:
\end{equation}
for suitable numbers $C^{[n]}_k\in\C$,
and determine the sequence of polynomials 
\begin{equation}\label{eq: pn(x)}
p_n(x)
=\sum_{k=1}^nC^{[n]}_kx^k
\:,
\end{equation}
for $n\geq 1$, using the ordinary multiplicative structure on $\mathscr{Z}^{\ostar_+k}$.

\begin{theorem}\label{theorem: transformation formula from the g(k) to Z ostar n}
For $n\geq 0$, $g^{(n)}$ define elements in $V\cap\SK_{n+1}$.
\begin{enumerate}
\item 
For $n\geq 1$, 
we have
\begin{equation}\label{eq: writing convolutions of Z in terms of the g(k)}
\mathscr{Z}^{\ostar_+n}(x)
=\sum_{k=0}^nc^{[n]}_{k}g^{(k)}(x)
\:.
\end{equation}
The coefficients for $1\leq k\leq n$ are given by
\begin{equation*}
c^{[n]}_k
=(-1)^{n-1}(\Delta\mathscr{Z})^{n-k}\frac{(k-1)!}{(n-1)!}\:s(n,k)
\:,
\end{equation*}
where $s(n,k)$ are the signed Stirling numbers of the first kind.
For $k=0$, we have $c^{[0]}_0=1$ and
\begin{equation*}
c^{[n]}_{0}
=(\Delta\mathscr{Z})^n\left\{\frac{1}{2^n}
+\frac{(-1)^{n+1}}{n!}\sum_{m=1}^{n+1}\frac{s(n+1,m)}{m}\right\}\:,
\quad n\geq 1
\:.
\end{equation*}
\item
For $n\geq 0$, the inverse transformation (\ref{eq: g(n) expressed in terms of Z ostar k})
is given, for $1\leq k\leq n$, by
\begin{equation*}
C^{[n]}_k
=(-1)^{k-1}(\Delta\mathscr{Z})^{n-k}\frac{(k-1)!}{(n-1)!}S(n,k)
\:.
\end{equation*}
Here $S(n,k)$ are the Stirling numbers of the second kind.
A closed formula is
\begin{equation*}
C^{[n]}_k
=\frac{(-1)^{k-1}}{(n-1)!}(\Delta\mathscr{Z})^{n-k}
\sum_{t=0}^{k-1}(-1)^t(k-t)^{n-1}\binom{k-1}{t}
\:.
\end{equation*} 
The generating function for the associated sequence $p_n$ of polynomials eq.\ (\ref{eq: pn(x)})
for $n\geq 1$ is given by
\begin{equation}\label{def: generating function of the pn}
G(x,y)
:=\frac{x\cdot\exp(y)}{1-x+x\cdot\exp(y)}
\:.
\end{equation}
Moreover, we have
\begin{equation*}
C^{[n]}_0
=\frac{(\Delta\mathscr{Z})^n}{n!}2B_{n}(1-2^{n-1})
\:,
\quad
n\geq 0
\:.
\end{equation*}
(Here $B_0=1$ and for $n\geq 2$, $B_n$ is the $n$\scr{th} Bernoulli number.)
\end{enumerate}
\end{theorem}

\begin{proof}
We have $F(x,y)\in V$, so $g^{(n)}\in V$ for $n\geq 0$.
The fact that $g^{(n)}\in K_{n+1}$ follows from the comparison with $\mathscr{Z}^{\ostar_+n}$.
\begin{enumerate}
 \item 
Comparing the residues from eqs (\ref{eq: Laurent series coefficient of Z oreg N to first order pole at u tau}) and (\ref{eq: residue of g(n) at m tau})
shows that
\begin{equation*}
s(n,k)
=(-1)^{n-1}(\Delta\mathscr{Z})^{k-n}\frac{(n-1)!}{(k-1)!}c^{[n]}_{k}
\:
\end{equation*}
for $1\leq k\leq n$
satisfy equation defining the signed Stirling numbers of the first kind \cite[p.\ 824]{AS:1965}
\begin{equation}\label{eq: definition of the generating function of the Stirling numbers of the first kind}
(m-1)\ldots(m-n+1)
=\sum_{k=1}^ns(n,k)m^{k-1}
\:.
\end{equation}
In particular, we have 
\begin{equation}\label{eq: matrix equation for transforming the Delta g into the ostar Delta convolutions of Z}
\Delta{\mathscr{Z}^{\ostar_+(n+1)}}(x)
=\sum_{k=1}^{n+1}c^{[n+1]}_{k}\Delta{g^{(k)}}(x)
\:.
\end{equation}
For $n=1$ we obtain $(1\ostar_+\mathscr{Z})=(\Delta\mathscr{Z})/2$ by comparison with eq.\ (\ref{eq: Delta Z ostar+ n}).
Solving the latter equation for $\mathscr{Z}^{\ostar_+n}$
and using eqs (\ref{eq: Delta g(k) as sum over g(k-ell)}) and (\ref{eq: matrix equation for transforming the Delta g into the ostar Delta  convolutions of Z}),
yields
\begin{equation*}
\mathscr{Z}^{\ostar_+n}(x)
=\sum_{j=0}^n\left(\left(\frac{\Delta\mathscr{Z}}{2}\right)^n\delta_{j0}
-\sum_{k=j+1}^{n+1}c^{[n+1]}_{k}\frac{(\Delta\mathscr{Z})^{k-j-1}}{(k-j)!}\right)g^{(j)}(x)
\:.
\end{equation*}
This shows 
\begin{equation*}
c^{[n]}_{0}
=\left(\frac{\Delta\mathscr{Z}}{2}\right)^n
-\sum_{k=1}^{n+1}c^{[n+1]}_{k}\frac{(\Delta\mathscr{Z})^{k-1}}{k!}
\:.
\end{equation*}
\item
The coefficients $C^{[n]}_k$ for $1\leq k\leq n$ of the inverse transformation are obtained by comparing residues in a similar way as above.
Thus 
\begin{equation*}
G(x,y)
=\sum_{n=1}^{\infty}p_n(x)y^{n-1}
=\sum_{n=1}^{\infty}\sum_{k=1}^{\infty}(-1)^{k-1}\frac{(k-1)!}{(n-1)!}x^ky^{n-1}S(n,k)
\:.
\end{equation*}
Here $S(n,k)=0$ for $k>n$, so the inner sum on the r.h.s.\ is actually finite.
Exchanging the order of summation, inserting $S(n,k)=kS(n-1,k)+S(n-1,k-1)$ for $1\leq k\leq n$ \cite[p.\ 825]{AS:1965} and performing an index shift, 
yields
\begin{equation*}
G(x,y)
=\sum_{k=1}^{\infty}(-1)^{k-1}x^k\sum_{n=0}^{\infty}\frac{y^n}{n!}\left(k!S(n,k)+(k-1)!S(n,k-1)\right)
\end{equation*}
Using $m!\sum_{n=m}^{\infty}S(n,m)\frac{y^n}{n!}=(e^y-1)^m$ \cite[p.\ 824]{AS:1965},
where as before the summation can be extended to $n\geq 0$ without change,
we obtain
\begin{equation*}
G(x,y)
=xe^y\sum_{k=1}^{\infty}(-x)^{k-1}(e^y-1)^{k-1}
=\frac{x\cdot\exp(y)}{1-x+x\cdot\exp(y)}
\:.
\end{equation*}
By eq.\ (\ref{eq: g(n) expressed in terms of Z ostar k}),
\begin{equation*}
g^{(n)}(x)-\sum_{k=1}^nC^{[n]}_k\mathscr{Z}^{\ostar_+k}(x)
=C(n,0)
\:.
\end{equation*}
The r.h.s.\ is invariant under integration over the real period, so
\begin{equation*}
\sum_{k=1}^nC^{[n]}_k\left(\mathscr{Z}^{\ostar_+k}-\lim_{\eps\searrow 0}\int_{i\eps}^{1+i\eps}\mathscr{Z}^{\ostar_+k}(x)dx\right)
=g^{(n)}-\lim_{\eps\searrow 0}\int_{i\eps}^{1+i\eps}g^{(n)}(x)dx
\:.
\end{equation*}
For $k\geq 1$, the $\mathscr{Z}^{\ostar_+k}$ integral equals $\left(\Delta\mathscr{Z}/2\right)^k$.
(Note that this also settles the $g^{(n)}$ integral for $n=1$.)
We are lead to computing $p_n(\Delta\mathscr{Z}/2)$ for $n\geq 1$.
We have 
$p_1(\Delta\mathscr{Z}/2)
=(\Delta\mathscr{Z})/2$, 
and for $n\geq 2$,
\begin{equation}\label{eq: pn((Delta Z)/2)}
p_{n}\left(\frac{\Delta\mathscr{Z}}{2}\right)
=(-\Delta\mathscr{Z})^n\frac{(2^n-1)}{n!}B_n
\:.
\end{equation}
Indeed,
\begin{equation*}
G\left(\frac{1}{2},y\right)
=1+\frac{2}{e^{2y}-1}-\frac{1}{e^y-1}
\:.
\end{equation*}
Since $\mathscr{Z}$ has residue equal to one at $x=0$, 
we have for $0<\eps<\Im(\tau)$,
\begin{equation}\label{eq: Z integral along the real p[eriod in the lower half plane}
\int_{-i\eps}^{1-i\eps}\mathscr{Z}(x)dx 
=-\frac{(\Delta\mathscr{Z})}{2} 
\:.
\end{equation}
We prove
\begin{equation}\label{eq: identity between g(k+1) integral and Bernoulli number}
\int_0^1g^{(\ell)}(x)dx
=\frac{B_{\ell}}{\ell!}(\Delta\mathscr{Z})^{\ell}\:,
\quad\ell\geq 2\:,
\end{equation}
by referring to the recursion relation for the Bernoulli numbers
\begin{equation}\label{eq: recursion relation for the Bernoulli numbers}
\sum_{\ell=0}^{k-1}\binom{k}{\ell}B_{\ell}
=0
\:,\quad k\geq 2\:,
\end{equation}
(for $k=1$, we have $B_0=1$).
For $k\geq 1$, 
we have by eq.\ (\ref{eq: Delta g(k) as sum over g(k-ell)}),
\begin{equation*}
\frac{k!}{(\Delta\mathscr{Z})^{k}}\Delta g^{(k)}(x)
=\sum_{\ell=0}^{k-1}\binom{k}{\ell}\frac{\ell!}{(\Delta\mathscr{Z})^{\ell}}g^{(\ell)}(x)
\:.
\end{equation*}
Moreover, for $k\geq 2$, 
\begin{equation*}
\int_{-i\Im(\tau)/2}^{1-i\Im(\tau)/2}\Delta g^{(k)}(x)dx
=2\pi i\Res_{x=0}\left[g^{(k)}(x)\right]
=0
\end{equation*}
by $1$-periodicity and regularity of $g^{(k)}(x)$ on $\R$.
So  for $k\geq 2$, we have the recursion relation
\begin{equation*}
0
=\sum_{\ell=0}^{k-1}\binom{k}{\ell}\frac{\ell!}{(\Delta\mathscr{Z})^{\ell}}\int_{-i\Im(\tau)/2}^{1-i\Im(\tau)/2}g^{(\ell)}(x)dx
\:.
\end{equation*}
(For $k=1$, the summand for $\ell=0$ equals $1$.)
By comparison with the rescursion relation (\ref{eq: recursion relation for the Bernoulli numbers}) 
and because of the identity (\ref{eq: Z integral along the real p[eriod in the lower half plane}), where $B_1=-\frac{1}{2}$,
we conclude that eq.\ (\ref{eq: identity between g(k+1) integral and Bernoulli number}) is true.
Eqs (\ref{eq: identity between g(k+1) integral and Bernoulli number}) and (\ref{eq: pn((Delta Z)/2)}) determine the additive constant in eq.\ (\ref{eq: g(n) expressed in terms of Z ostar k}).
\end{enumerate}
\end{proof}

The first few functions are given by $\mathscr{Z}=g^{(1)}$ and 
\begin{align*}
\mathscr{Z}^{\ostar_+2} 
\=-g^{(2)}+(\Delta\mathscr{Z})g^{(1)}
-\frac{(\Delta\mathscr{Z})^2}{6}\\
\mathscr{Z}^{\ostar_+3}
\=g^{(3)}-\frac{3}{2}(\Delta\mathscr{Z})g^{(2)}+(\Delta\mathscr{Z})^2g^{(1)}
-\frac{(\Delta\mathscr{Z})^3}{4}\nn\\
\mathscr{Z}^{\ostar_+4}
\=-g^{(4)}+2(\Delta\mathscr{Z})g^{(3)}-\frac{11}{6}(\Delta\mathscr{Z})^2g^{(2)}+(\Delta\mathscr{Z})^3g^{(1)}
-\frac{103}{360}(\Delta\mathscr{Z})^4\nn\\
\mathscr{Z}^{\ostar_+5}
\=g^{(5)}-\frac{5}{2}(\Delta\mathscr{Z})g^{(4)}+\frac{35}{12}(\Delta\mathscr{Z})^2g^{(3)}-\frac{25}{12}(\Delta\mathscr{Z})^3g^{(2)}+(\Delta\mathscr{Z})^4g^{(1)}
-\frac{43}{144}(\Delta\mathscr{Z})^5\nn
\end{align*}
respectively
\begin{align*}
g^{(2)}
\=-\mathscr{Z}^{\ostar_+2}
+(\Delta\mathscr{Z})\mathscr{Z}
-\frac{1}{6}(\Delta\mathscr{Z})^2\\
g^{(3)}
\=\mathscr{Z}^{\ostar_+3}
-\frac{3}{2}(\Delta\mathscr{Z})\mathscr{Z}^{\ostar_+2}
+\frac{1}{2}(\Delta\mathscr{Z})^2\mathscr{Z}\\
g^{(4)}
\=-\mathscr{Z}^{\ostar_+4}
+2(\Delta\mathscr{Z})\mathscr{Z}^{\ostar_+3}
-\frac{7}{6}(\Delta\mathscr{Z})^2\mathscr{Z}^{\ostar_+2}
+\frac{1}{6}(\Delta\mathscr{Z})^3\mathscr{Z}
+\frac{7}{360}(\Delta\mathscr{Z})^4\\
g^{(5)}
\=\mathscr{Z}^{\ostar_+5}
-\frac{5}{2}(\Delta\mathscr{Z})\mathscr{Z}^{\ostar_+4}
+\frac{25}{12}(\Delta\mathscr{Z})^2\mathscr{Z}^{\ostar_+3}
-\frac{5}{8}(\Delta\mathscr{Z})^3\mathscr{Z}^{\ostar_+2}
+\frac{1}{24}(\Delta\mathscr{Z})^4\mathscr{Z}
\:.
\end{align*}


\subsection{The associated convolution polynomials}

We consider the sequence of polynomials (\ref{eq: pn(x)}) 
which is naturally associated to the set of transformations (\ref{eq: g(n) expressed in terms of Z ostar k}).
To simplify notations, we set $\Delta\mathscr{Z}=1$. 

\begin{proposition}
For $n\geq 1$, let $p_n(x)\in\Q[x]$ be the sequence of polynomials 
\begin{equation*}
 p_n(x)
 =\sum_{k=1}^n(-1)^{k-1}\frac{(k-1)!}{(n-1)!}S(n,k)x^k
 \:,
\end{equation*}
where $S(n,k)$ are the Stirling numbers of the second kind.
\begin{enumerate}
\item\label{symmetry} We have $p_1(x)=x$, $p_2(x)=x(1-x)$ and 
\begin{equation}\label{eq: symmetry resp. antisymmetry of pn(x) under x mapsto 1-x}
p_n(1-x)
=(-1)^np_n(x)\:,\quad n\geq 2\:.
\end{equation}
\item For $n\geq 1$, $p_n(x)$ has $n$ simple zeroes $x_j^{[n]}$, $j=1,\ldots,n$, in the interval $[0,1]$ so that
$x_1^{[n]}=0$, $x_n^{[n]}=1$ and $x_j^{[n]}<x_{j+1}^{[n+1]}<x_{j+1}^{[n]}$
for $j=1,\ldots,n-1$.
Thus the zeroes of $p_n$ and of $p_{n+1}$ are interlaced.
\item\label{item: rho(x)} The distribution of zeroes $x\in(0,1)$ of $p_n$ for large $n$ is governed by the density function $\rho:(0,1)\rechts\R$
defined by
\begin{equation*}
\rho(x)
=\frac{1}{x(1-x)\:|y_0|^2}
\:.
\end{equation*}
Here $y_0=\Re\left(\log(x-1)-\log x\right)+i\pi$.
\item 
For $x\in(0,1)$ and for $n\geq 2$, 
we have
\begin{equation*}
p_{n}(x)
=\frac{-1}{(2\pi i)^n}\left\{\zeta\left(n,\frac{y_0}{2\pi i}\right)+(-1)^n\zeta\left(n,-\frac{y_0}{2\pi i}-1\right)\right\}
\:. 
\end{equation*}
Here $y_0$ is as in part \ref{item: rho(x)}, and $\zeta(n,a)$ is the Hurwitz zeta function.
\end{enumerate}
\end{proposition}

\begin{proof}
\begin{enumerate}
\item
The generating function $G(x,y)$ from eq.\ (\ref{def: generating function of the pn}) satisfies
\begin{equation}\label{eq: G(x,y)-x using coth y/2}
G(x,y)-x
=\frac{2x(1-x)}{2x-1+\coth\frac{y}{2}}
\end{equation}
so $G(1-x,-y)=1-G(x,y)$.
It follows that $p_1(1-x)=1-p_1(x)$ and eq.\ (\ref{eq: symmetry resp. antisymmetry of pn(x) under x mapsto 1-x}) holds.
\item
We show by induction that all zeroes are simple and real. 
The statement is true for $n=1,2$.
$G(x,y)$ satisfies the PDE $\partial G/\partial y=x(1-x)\:\partial G/\partial x$.
Thus for $n\geq 1$,
\begin{equation*}
np_{n+1}(x)
=x(1-x)p_n'(x)
\:.
\end{equation*}
It follows that for $n\geq 3$, between every pair of subsequent zeroes of $p_n(x)$ in $[0,1]$ there lies at least one real zero of $p_{n+1}(x)$,
in fact one simple zero since $\deg p_n'=\deg p_n-1$.
By the fact that $S(n,1)=1$, $p_n'(x)$ does not vanish at $x=0$, so $p_{n+1}$ has two additional simple zeroes at $x=0,1$.
\item
By eq.\ (\ref{eq: G(x,y)-x using coth y/2}), for every $x\in(0,1)$, $G(x,y)$  has a simple pole at
\begin{equation*}
y_k
=\Re\left(\log\frac{x-1}{x}\right)
+\pi i(2k+1)\:,
\quad k\in\Z\:,
\end{equation*}
with residue $1$, since $\frac{d}{dy}(2x-1+\coth\frac{y}{2})|_{y=y_k}=2x(1-x)$.
Choose $\eps>0$ so that $G(x,y)$ is regular on the region enclosed by the contour $|y|=\eps$.
For $m\geq 0$ and $n\geq 1$,
\begin{equation}\label{eq: asymptotic behaviour of pn(x) for large n}
p_{n}(x)
=\ointctrclockwise_{|y|=|y_m|+\eps}\frac{G(x,y)}{y^{n}}\frac{dy}{2\pi i}
-\sum_{k=-m-1}^m\frac{1}{y_k^{n}}
\:.
\end{equation}
The integral tends to zero for $n\rechts\infty$ and the leading term in this limit is given by
\begin{equation*}
\frac{1}{y_{-1}^{n}}+ \frac{1}{y_0^{n}}
=\frac{2\cos(n\varphi)}{r^n}
\:.
\end{equation*}
Here we have in polar coordinates, $y_0=r\exp(i\varphi)=\bar{y}_{-1}$.
In order for $x$ to be a root of $p_{n}$ for large $n$, $\varphi$ must be an odd multiple of $\pi/(2n)$,
so $d\varphi\sim\pi/n$. 
Moreover,
\begin{equation*}
\frac{d\cot\varphi}{dx} 
=\frac{1}{\pi}\frac{1}{x(x-1)}\:,
\quad
\frac{d\cot\varphi}{d\varphi}
=-\frac{1}{1-\cot^2\varphi}
\:.
\end{equation*}
So the measure in the large $n$ limit on $(0,1)$ equals
\begin{equation*}
dx
\approx\frac{\pi}{n}\frac{dx}{d\varphi} 
=x(1-x)\frac{|y_0|^2}{n}
\end{equation*}
where $1/n$ is the point measure of a single zero of $p_n$.
\item
For $x\in(0,1)$ and for $n\geq 1$, we have 
\begin{equation*}
\frac{G(x,y)}{y^{n}}
=\frac{\frac{x}{1-x}e^y}{(1+\frac{x}{1-x}e^y)y^{n}}
\:\longrightarrow\:0\quad\text{for $|y|\rechts\infty$}
\end{equation*}
So taking in eq.\ (\ref{eq: asymptotic behaviour of pn(x) for large n}) for $n\geq 2$ the limit $m\rechts\infty$,
\begin{equation*}
p_{n}(x)
=-\sum_{k=-\infty}^{\infty}\frac{1}{y_k^{n}}
\:,
\end{equation*}
yields the claimed identity.
\end{enumerate}
\end{proof}

To polynomial associated to the inverse transformation (\ref{eq: writing convolutions of Z in terms of the g(k)}) 
only reproduces eq.\ (\ref{eq: definition of the generating function of the Stirling numbers of the first kind}).

\section{Conclusion and outlook}

We have generalized the standard convolution for regular periodic functions on $\R$ 
to a commutative and associative operation on arbitrary quasi-elliptic functions,
which defines a second algebra structure on the space of such functions.
Moreover, we provide a method for computing iterated convolutions of elliptic functions explicitly,
which answers a question raised in \cite{L:2018}.
Our approach naturally leads to the discussion of a family of polynomials with attractive features,
closely related to the Eulerian polynomials but of greater transparency.
There certainly is a prospect of applications in various areas.

This paper settles the case for indecomposable bundles of degree zero over elliptic curves.
The treatment of convolutions of sections in more general vector bundles is left for future work.
The first case to handle are mock Jacobi forms.

It is conceivable that the convolution defines an operation on the moduli space of vector bundles.
It will be interesting to study analogous structures in higher genus,
where translations on the elliptic curve are replaced by transations in the Jacobian.
This is related to general period integrals in the sense of Kontsevich and Zagier.

\subsection{Acknowledgement}

The author thanks W.\ Nahm for discussions in the early and final stages of this work.
This work is funded by a Government of Ireland Postdoctoral Award 2018/583 from the Irish Research Council.

 \end{document}